\documentclass[a4paper,11pt,twoside]{article}

\usepackage[top=28mm,bottom=28mm]{geometry}
\usepackage[english]{babel}
\usepackage[utf8]{inputenc}
\usepackage{times}
\usepackage{amsmath,amssymb,amsfonts,amsthm}
\usepackage{setspace}
\usepackage{color}
\usepackage{paralist}
\usepackage{pgfplots}
\usepackage{mathrsfs}
\usepackage{mathtools}
\usepackage{bm}
\usepackage{tikz-cd}
\usepackage{tikz}
\usetikzlibrary{chains,
  matrix,
  positioning,
  scopes,arrows
}
\usepackage{pdfpages}
\usepackage{bibentry}
\usepackage{multirow}
\usepackage[bookmarks=false]{hyperref}

\newtheorem{theorem}{Theorem}[section]
\newtheorem{lemma}[theorem]{Lemma}
\newtheorem{proposition}[theorem]{Proposition}
\newtheorem{corollary}[theorem]{Corollary}
\theoremstyle{definition}
\newtheorem{definition}[theorem]{Definition}
\newtheorem{example}[theorem]{Example}
\newtheorem{notation}[theorem]{Notation}
\theoremstyle{remark}
\newtheorem{remark}[theorem]{Remark}

\newcommand{\Z}{{\mathbb Z}}

\newcommand{\C}{{\mathbb C}}

\newcommand{\ts}{{\tilde{S}}}

\newcommand{\Osh}{{\cal O}}
\newcommand{\resdot}{\scriptscriptstyle \bullet}

\newcommand{\tenf}{{\textbf{F} \otimes \textbf{F}}}
\renewcommand{\P}{{\mathbb P}}

\DeclareMathOperator{\im}{im}
\DeclareMathOperator{\depth}{depth}

\DeclareMathOperator{\coker}{coker}
\DeclareMathOperator{\Ext}{Ext}

\DeclareMathOperator{\grade}{grade}
\DeclareMathOperator{\pd}{projdim}

\DeclareMathOperator{\ann}{ann}
\DeclareMathOperator{\Proj}{Proj}
\DeclareMathOperator{\Spec}{Spec}
\DeclareMathOperator{\Hom}{Hom}

\DeclareMathOperator{\rank}{rank}
\DeclareMathOperator{\id}{id}

\AtEndDocument{
  \par
  \medskip
  \begin{tabular}{@{}l@{}}
    \textsc{I. Stenger, Universit\"{a}t des Saarlandes, Campus E2 4,} \\
     \textsc{66123 Saarbr\"{u}cken, Germany.}\\
    \textit{E-mail address}: \texttt{stenger@math.uni-sb.de}
  \end{tabular}}

\begin{document}
\title{A structure result for Gorenstein algebras of odd codimension}  
\author{Isabel Stenger}
\date{\today}

\maketitle  
\begin{abstract}The famous structure theorem of Buchsbaum and Eisenbud gives a complete characterization of Gorenstein ideals of codimension 3 and their minimal free resolutions.  We generalize the ideas of Buchsbaum and Eisenbud from Gorenstein ideals to Gorenstein algebras and present a description of Gorenstein algebras of any odd codimension. As an application we study the canonical ring of a numerical Godeaux surface.
\end{abstract}

\section{Introduction}
The description of Gorenstein rings plays a central role in algebra and geometry. 
Gorenstein ideals in codimension 1 and 2 are known to be complete intersections. Buchsbaum and Eisenbud (\cite{BuchsbaumEisenbud77}) showed that any Gorenstein ideal of codimension 3 is generated by $2k+1$ elements, $k\geq 1$, which are the submaximal Pfaffians of a skew-symmetric matrix of size $2k+1$. Codimension 4 Gorenstein ideals are not fully understood but there are partial characterization results due to Kustin-Miller (\cite{KustinMiller82}) and Reid (\cite{Reid13}). 

The non-existence of any structure theorem for Gorenstein ideals in higher codimension is a main motivation for the following generalization. Let $I$ be a Gorenstein ideal of codimension $d \geq 4$ in a regular local ring or a graded polynomial ring $T$. Instead of considering $R = T/I$ as an $T$-module and describing its minimal free resolution, the idea is to reduce the codimension by considering $T$ as an algebra over some other regular local ring $S$. More precisely, let $\varphi\colon S \rightarrow R$ be a ring homomorphism which induces on $R$ the structure of a finitely generated $S$-module. Let us further assume that $R$ is a perfect Cohen-Macaulay $S$-module of codimension $ c= \dim S - \dim_S R \leq 3$. Then the Auslander-Buchsbaum formula implies that the minimal free resolution of $R$ as an $S$-module has length $\leq 3$ and the question is whether there exist a general description of such a minimal free resolution.  In codimension 1, \cite{CataneseRegular} gives a characterization for the case that $R$ is the canonical ring of a regular surface of general type which holds also in greater generality.  
 Gorenstein algebras of codimension 2 and their special symmetric minimal free resolutions were described in \cite{Grassi96} and \cite{Boehning05}. 
 
 In this article, we focus on Gorenstein algebras whose codimension is an odd number and adapt the ideas of Buchsbaum and Eisenbud to Gorenstein algebras after imposing some mild additional condition on the homomorphism $\varphi$. 

 Throughout this article, all considered rings are commutative Noetherian with identity. To start with, we briefly recall some notation. 
 \begin{definition}[\cite{BuchsbaumEisenbud77}, Section 2] \label{def_alternating}
   Let $S$ be a ring, and let $F$ be a finitely generated free $S$-module. We call a map $f\colon F \rightarrow F^{\ast} := \Hom_S(F,S)$ \textit{alternating} if, with respect to some (and hence any) basis and dual basis of $F$ and $F^{\ast}$ respectively, the matrix corresponding to $f$ is skew-symmetric, or equivalently, if $f^\ast = -f$.  
\end{definition}

\begin{definition}\label{def_perfect}[\cite{BrunsHerzog98}, Chapter 1]
Let $S$ be a ring and let $M$ be a finitely generated $S$-module. Then the $\textit{grade}$ of $M$ is defined as 
\[ \grade M =  \min\{i \mid \Ext^i_S(M,S) \neq 0\}. \]
Moreover, we call $M$ \textit{perfect}, if $\grade M = \pd_S M$. 
\end{definition}

\begin{definition} Let $(S,\mathfrak{m})$ be a regular local ring, and let $R$ be a ring. Moreover, let $\varphi\colon S \rightarrow R$ be a homomorphism of rings which induces on $R$ the structure of a finitely generated $S$-module. The number $c = \dim S -  \dim_S R$ is called the \textit{codimension} of $R$. We call $R$ a \textit{Gorenstein ($S$-)algebra} or \textit{relatively Gorenstein of codimension $c$} if $R$ is a perfect $S$-module and 
  \begin{equation*}
  R \cong \Ext^c_S(R,S) 
  \end{equation*} 
  as $R$-modules.
\end{definition}
\begin{remark}\label{rem_modulestructureExt}
	Note that the functoriality of $\Ext_S^{c}(,S)$ induces the structure of an $R$-module on $\Ext^c_S(R,S)$. Indeed if $r \in R$  and $\mu_r$ denotes the multiplication map on $R$ by $r$, then, for any $g \in \Ext_S^{c}(R,S)$, we set $rg =  \Ext_S(\mu_r,S)(g) \in \Ext_S^{c}(R,S)$.
\end{remark}
    
The main result of this article is the following:
\begin{theorem}\label{thm_mainthm}
	Let $(S,\mathfrak{m})$ be a regular local ring, and let $R$ be any ring. Let $\varphi \colon S \rightarrow R$ be a ring homomorphism which induces on $R$ the structure of a finitely generated  $S$-module. Furthermore, set $\tilde{S}  = S/\ann_S R$, and let $\tilde{\varphi}\colon \tilde{S} \rightarrow R$ be the induced injective homomorphism. Assume that the condition 
	$$ (\diamondsuit) \text{\quad  $\tilde{\varphi}_p \colon \tilde{S}_p \rightarrow R_p$ is an isomorphism for all minimal primes of $\tilde{S}$ }$$
	holds.
	Let  $c  = \dim S - \dim_S R$ be the codimension. If $c = 2m+1$ for some $m\geq 0$, 
	the following conditions are equivalent:
	\begin{itemize}
	\item $R$ is a Gorenstein algebra of codimension $c$.
	\item There exists a minimal free resolution of $R$ as an $S$-module of type
	\begin{equation}\label{eq_minres}
	 0 \leftarrow R \xleftarrow{d_0}
	 F_0  \xleftarrow{d_1} F_1 \leftarrow \cdots \xleftarrow{d_{m}}  F_m \xleftarrow{d_{m+1}} F_m^\ast \xleftarrow{d_{m}^\ast} \cdots \leftarrow F_1^\ast \xleftarrow{d_1^\ast} F_0^\ast \leftarrow 0,
	\end{equation}
	where $F_0  =  S \oplus  F_0' $ for some free module $F_0'$, $d_0 = ( \varphi  \mid  d_0')$ and $d_{m+1}^\ast = (-1)^m d_{m+1}$. 
	\end{itemize}
\end{theorem}
\begin{notation} Let $S,R$ and $\varphi$ be as in Theorem \ref{thm_mainthm}.  In the following, we denote by $e_0$ the element of $F_0$ corresponding to $1_S \in S \hookrightarrow F_0$. Thus, by the choice of $d_0$ in \eqref{eq_minres}, we have $d_0(e_0) = \varphi(1_S) = 1_R$. 
\end{notation}

\begin{remark} Let us see how Theorem \ref{thm_mainthm} recognizes the known results in codimension 1 and 3. 
So let us assume that $R = S/I$ for some ideal $I \subseteq S$ of codimension $c$. Then the condition $(\diamondsuit)$ is trivially satisfied. If $c=1$, then $d_1\colon F_0 = S \leftarrow F_1$ is a symmetric matrix, and hence $I$ is generated by one element. If $c = 3$, then $\rank F_1 = 1+ \rank d_2$ is an odd number as any skew-symmetric matrix has an even rank and the result reduces to (a reformulation of) the structure theorem of Buchsbaum-Eisenbud for codimension 3 Gorenstein ideals.

\end{remark}
\begin{remark} Note that the equivalent conditions on $R$ in Theorem \ref{thm_mainthm} imply that $R$ is a Cohen-Macaulay $S$-module. Indeed, by the Auslander-Buchsbaum formula we have
	$$ \depth R = \depth S - \pd_S(R) = \dim S - \pd_S(R) = \dim_S(R).$$  
\end{remark}
\begin{remark} 
Note that if $R$ is an integral domain, then so is $\tilde{S}$, and the fact that $R$ is a finitely generated $\tilde{S}$-module implies that the condition $(\diamondsuit)$ is equivalent to $Q(\tilde{S}) \cong Q(R)$, or geometrically that the corresponding integral affine schemes are birational.   
\end{remark}
Let us first show how condition $(\diamondsuit)$ implies that every $S$-linear homomorphism $g\colon R \rightarrow R$ is also $R$-linear, more precisely is the multiplication by an element in $R$:
\begin{proposition}\label{prop_iso}
	Let $S$ be a local Cohen-Macaulay ring and let
	$\varphi\colon S \rightarrow R$ be a ring homomorphism which induces on $R$ the structure of a finitely generated Cohen-Macaulay $S$-module. Furthermore, let $I = \ann_S R$ and set $\tilde{S} = S/I$. Assume that the induced homomorphism  $\tilde{S}_p \rightarrow R_p$ is an isomorphism for every minimal prime ideal of $\tilde{S}$.
	Then the natural homomorphism $R \rightarrow \Hom_S(R,R)$ is an isomorphism. 
\end{proposition}
\begin{proof} First notice that any associated prime of $R$ as an $S$-module contains $I$. Thus, the minimal primes of $\ts$ coincides with the minimal primes of $R$ as an $S$-module (via the identification $\Spec(\ts) \rightarrow \Spec(S)$). On the other hand, $R$ being a Cohen-Macaulay $S$-module implies that every associated prime of $R$ as an $S$-module is minimal. Consequently, $I$ is unmixed and every associated prime is of height $\dim S - \dim R$.

For a minimal prime $p \subset \tilde{S}$, we have $\tilde{S}_p \cong R_p$ as $\tilde{S}_p$-modules by assumption. Thus,  
\begin{equation}\label{eq_min}
\Hom_{\tilde{S}_p}(R_p,R_p) \cong \tilde{S}_p \cong R_p. 
\end{equation}
 Now let $\psi\colon R \rightarrow R$ be an arbitrary $S$-linear homomorphism and set $\alpha :=  \psi-\psi(1)\id_R$. 
We want to show that $\alpha$ is the zero homomorphism.
	 The module $N = \im(\alpha)$ is an $\tilde{S}$-submodule of $R$ and for all minimal primes of $\ts$ we have $N_p = 0$ by \eqref{eq_min}. Suppose that $N \neq 0$. Then $N$ has an associated prime $q$ and $N_q \neq 0$. But $q$ also being an associated prime of $R$ and hence a minimal prime of $\ts$
	 gives a contradiction.
\end{proof}
\begin{remark}\label{rem_isob}
	 Let $S, R$ and $\varphi$ be as in the setting of Proposition \ref{prop_iso}. Then the statement remains correct if we consider homomorphisms from $R$ to an isomorphic $\tilde{S}$-module. More precisely, let $B$ be an $R$-module which is isomorphic to $R$ as an $\tilde{S}$-module, then the homomorphism 
	$$ B \rightarrow \Hom_{\ts}(R,B)$$ sending an element $b \in B$ to the multiplication homomorphism by $b$ is an isomorphism. 
\end{remark}
Using Proposition \ref{prop_iso} we are now able to prove one part of Theorem \ref{thm_mainthm}: 
\begin{proof} [Proof of Theorem \ref{thm_mainthm}, "(skew-) symmetric resolution $\Rightarrow$ Gorenstein"]

Let 
\[
  0 \leftarrow R \xleftarrow{d_0}
  F_0  \xleftarrow{d_1} F_1 \leftarrow \cdots \xleftarrow{d_{m}}  F_m \xleftarrow{d_{m+1}} F_m^\ast \xleftarrow{d_{m}^\ast} \cdots \leftarrow F_1^\ast \xleftarrow{d_1^\ast} F_0^\ast \leftarrow 0,
\]
 be a minimal free resolution of $R$ as an $S$-module with $d_{m+1}^\ast = (-1)^m d_{m+1}$. By assumption we have 
    $$ \Ext^i_S(R,S) = 0 \text{ for all } i < \depth(S) - \dim_S R  =  \dim S - \dim_S R = c. $$
   Thus, applying the functor $\Hom_S(,S)$ to the given minimal free resolution of $R$, we obtain a complex 
  \[
    0 \leftarrow \Ext^c_S(R,S) \leftarrow
    F_0  \xleftarrow{d_1} F_1 \leftarrow \cdots \xleftarrow{d_{m}}  F_m \xleftarrow{(-1)^m d_{m+1}} F_m^\ast \xleftarrow{d_{m}^\ast} \cdots \leftarrow  F_1^\ast \xleftarrow{d_1^\ast} F_0^\ast \leftarrow 0,
  \]
    which is a minimal free resolution of $\Ext^c_S(R,S)$. 
Comparing these two complexes, we construct a commutative diagram
\\
{\centering
	\begin{tikzpicture}[scale=\textwidth]
	\matrix(m)[matrix of math nodes,
	row sep=3.1em, column sep=2.4em,
	text height=1.1ex, text depth=0.21ex]
	{ 0 & R & F_0  & \cdots & F_m & F_m^\ast & \cdots  & F_0^\ast & 0  \\
	  0 & \Ext^c_S(R,S) & F_0  & \cdots & F_m & F_m^\ast & \cdots &   F_0^\ast & 0 \\};
	\path[->,font=\scriptsize,>=angle 90]
	(m-1-2) edge (m-1-1)
	        edge[dashed] node[left] {\tiny{$u$}} (m-2-2)
	(m-1-3) edge node[above] {\tiny{$d_0$}} (m-1-2)
	(m-1-4) edge node[above] {\tiny{$d_1$}}(m-1-3)
	(m-1-5) edge (m-1-4)
	(m-1-6) edge (m-1-5)
	(m-1-5) edge node[left] {\tiny{$\id_{F_m}$}} (m-2-5)  
	(m-2-2) edge (m-2-1)
	(m-2-3) edge (m-2-2)
	(m-2-6) edge (m-2-5)
	(m-2-5) edge  (m-2-4) 
	(m-2-4) edge node[above] {\tiny{$d_1$}}  (m-2-3)   
	(m-1-3) edge node[left] {\tiny{$\id_{F_0}$}} (m-2-3) 
	(m-2-6) edge node[above] {\tiny{$(-1)^m d_{m+1}$}} (m-2-5)  
    (m-1-6) edge node[above] {\tiny{$d_{m+1}$}}(m-1-5)  
    	    edge node[right] {\tiny{$(-1)^m\id_{F_m^\ast}$}} (m-2-6)  
    (m-1-7) edge (m-1-6) 	 
    (m-2-7) edge (m-2-6)  	 
    (m-1-8) edge (m-1-7)
     	    edge node[right] {\tiny{$(-1)^m\id_{F_0^\ast}$}} (m-2-8)  
    (m-2-8) edge (m-2-7)  
    (m-1-9) edge (m-1-8) 	 
    (m-2-9) edge (m-2-8)     	      
	;
	\end{tikzpicture} 
}
where $u$ is the $S$-linear isomorphism $u\colon R \rightarrow \Ext_{S}^{c}(R,S)$ induced by the chain maps. Now, since $u$ is also $\tilde{S}$-linear,
Proposition \ref{prop_iso} and Remark \ref{rem_isob} imply that $u$ is an $R$-linear homomorphism. Hence, 
\begin{equation*}\label{eq_extring}
R \cong 
\Ext_{S}^{c}(R,S),
\end{equation*} showing that $R$ is a Gorenstein $S$-algebra. 
\end{proof}
Next let us assume that $R$ is a Gorenstein algebra of odd codimension $c = 2m+1$. Then the minimal free resolution of $R$ as an $S$-module is of the form
 \[  0 \leftarrow R \leftarrow F_0 \xleftarrow{d_1} F_1 \leftarrow \cdots \leftarrow F_{c-1} \xleftarrow{d_c} F_c \leftarrow 0 \]
   Applying the functor $\Hom_S(,S)$ to the minimal free resolution of $R$, we obtain a complex 
    $$ 0 \rightarrow \Ext^c_S(R,S) \leftarrow F_c^\ast \xleftarrow{d_c^\ast} F_{c-1}^\ast \leftarrow \cdots \leftarrow  F_0^\ast \leftarrow 0$$ which is exact by the assumption on $R$. Moreover, as $R$ is a Gorenstein S-algebra, there exists an isomorphism $\tilde{\psi} 
    \colon \Ext^c_S(R,S) \rightarrow R$ which lifts to isomorphisms $\psi_i\colon F_{c-i} ^\ast \rightarrow F_i$. 
Using this chain isomorphism, we get a minimal free resolution of $R$ of the form 
\[
  0 \leftarrow R \xleftarrow{d_0}
  F_0  \xleftarrow{d_1} F_1 \leftarrow \cdots \xleftarrow{d_{m}}  F_m \xleftarrow{d_{m+1}\circ \psi_{m+1}} F_m^\ast \xleftarrow{d_{m}^\ast} \cdots \leftarrow F_1^\ast \xleftarrow{d_1^\ast} F_0^\ast \leftarrow 0.
\]
   Now, similar as in the proof of the Buchsbaum-Eisenbud structure theorem, the aim is to show the existence of an isomorphism of complexes which yields a symmetric respectively skew-symmetric map in the middle.  To do so, we first introduce a (skew-)commutative multiplication on the minimal free resolution $\textbf{F}  =\bigoplus F_i$ of $R$. Afterwards, we use this multiplication to define a map of chain complexes between the resolution and its dual. The last step is to show that this chain map is indeed an isomorphism of complexes. 
   
   \section{A Multiplicative Structure on the Minimal Free Resolution}
  
   Let $S,R$ and $\varphi\colon S \rightarrow R$ be as in Theorem \ref{thm_mainthm}. The assumption $(\diamondsuit)$ is not necessary for the following. 
   Recall that we denote by $e_0$ the element in $F_0$ corresponding to $1_S$ under the inclusion $S \hookrightarrow F_0$. 
 From now on $\textbf{F} = (F_{\resdot},d)$ will denote a (fixed) minimal free resolution of $R$ as an $S$-module
  \[  0 \leftarrow R \leftarrow F_0 \xleftarrow{d_1} F_1 \leftarrow \cdots  \leftarrow F_{c-1} \xleftarrow{d_c} F_c \leftarrow 0. \]
Moreover, we consider $\textbf{F}$ as an graded $S$-module by calling an element $x\in F_i$ homogeneous of (homological) degree $|x| = i$. 
We denote by $\textbf{F} \otimes \textbf{F}$ the tensor product of $\textbf{F}$ by itself
with differentials 
$$\delta(x \otimes y)  = d(x) \otimes y + (-1)^{|x|}x \otimes d(y)$$
for homogeneous elements $x,y \in \textbf{F}$.

 \begin{proposition}\label{prop_multiplication}
 There exists a chain map $\mu\colon \textbf{F} \otimes \textbf{F} \rightarrow \textbf{F}$ such that, writing $ab$ for $\mu(a\otimes b)$, the following holds:
 \begin{enumerate}[(i)]
 \item $d(fg) = d(f)g + (-1)^{|f|}fd(g)$ for any $f,g \in \textbf{F}$ homogeneous,\label{item:prop1}
 \item $\mu$ is homotopy-associative,\label{item:prop2}
 \item the element $e_0 \in F_0$ acts as a unit for $\mu$ on $\textbf{F}$, that means $e_0g = g = ge_0 $ for any $g \in \textbf{F}$, \label{item:unitelement}
 \item $fg = (-1)^{|f|\cdot|g|}gf$ for any $f,g \in \textbf{F}$ homogeneous, that means $\mu$ is \textit{commutative}. \label{item:commut}
  \end{enumerate} 
 \end{proposition}
We will prove this statement in several steps. Let $\tilde{\mu}: R \otimes_S R \rightarrow R$ be the multiplication homomorphism.  First we consider the following diagram
\begin{equation*}
 \begin{tikzpicture}
 \matrix(m)[matrix of math nodes,
 row sep=2.9em, column sep=2.9em,
 text height=1.5ex, text depth=0.25ex]
 {0 & R \otimes R  & (\tenf)_0 & (\tenf)_1 
 & \cdots \\
 0 & R &  F_0 & F_1  & \cdots \\};
 \path[->,font=\scriptsize,>=angle 90]
 (m-1-2) edge (m-1-1)
         edge node[left] {$\tilde{\mu}$} (m-2-2) 
 (m-1-3) edge node[above] {$d_0 \otimes d_0$}  (m-1-2) 
 (m-1-4) edge node[above] {$\delta_1$}  (m-1-3)
 (m-1-5) edge node[above] {$\delta_2$} (m-1-4)
 (m-2-2) edge (m-2-1)
 (m-2-3) edge node[above] {$d_0$} (m-2-2) 
 (m-2-4) edge node[above] {$d_1$} (m-2-3)
 (m-2-5) edge node[above] {$d_2$} (m-2-4)
 ;
 \end{tikzpicture}
\end{equation*}
where the first row is a complex and the second row is exact.
\begin{lemma}\label{lem_liftmultiplication}
Any chain map $\mu\colon \tenf \rightarrow \textbf{F}$ which is a lift of $\tilde{\mu}$ satisfies properties \eqref{item:prop1}-\eqref{item:prop2} of Proposition \ref{prop_multiplication}.
 \end{lemma}
\begin{proof} 
Property \eqref{item:prop1} is just the commutativity of the diagram above  with the induced chain map $\mu$. To verify the second property we have to show that the chain map 
$$ \rho \ := \mu \circ (\mu \otimes \id_\textbf{F} ) - \mu \circ (\id_\textbf{F}  \otimes \mu )\colon \textbf{F}^{\otimes 3}  \rightarrow \textbf{F}$$ is homotopic to zero. But $\rho$ is a lift of the map $\tilde{\mu} \circ (\tilde{\mu} \otimes \id_{R} ) - \tilde{\mu} \circ (\id_{R} \otimes \tilde{\mu})\colon R^{\otimes 3} \rightarrow R$ which is the zero map since $R$ is associative. Thus $\rho$ is homotopic to zero.
 \end{proof}
 \begin{remark} Similarly one can check that the map $\mu$ in Lemma \ref{lem_liftmultiplication} is homotopy-commutative. 
 \end{remark}
So it remains to show that there is a lift 
$\mu$ of  $\tilde{\mu}$ which satisfies also properties \eqref{item:unitelement} and \eqref{item:commut}.
To do this, we will first introduce the symmetric square of the complex $\textbf{F}$ as in \cite{BuchsbaumEisenbud77}. 
Let $M$ be the submodule of $\textbf{F} \otimes \textbf{F}$ generated by the relations
$$ \{f \otimes g - (-1)^{|f| \cdot |g|}g \otimes f \mid f,g \in \textbf{F} \textup{ homogeneous} \}.$$ Since $\delta(M) \subseteq M$, the module
\[ S_2(\textbf{F})  = (\textbf{F} \otimes \textbf{F})/ M \] 
inherits the structure of a complex of $S$-modules (with induced differentials $\bar{\delta}$). 
Let $n \geq 0$ and set $V = \oplus_{i + j = n, \ i<j} F_i \otimes F_{j}$. Then 
\[ S_2(\textbf{F})_n \cong
\begin{cases}
  V  & \text{if $n$ is odd}, \\
  V \oplus \bigwedge^{2}(F_{n/2}) & \text{if } n \equiv 2 \mod 4 , \\
  V \oplus S_2(F_{n/2}) & \text{if } n \equiv 0 \mod 4.
\end{cases} \]
In particular, $S_2(\textbf{F})$ is a complex of free $S$-modules. 
Let $\pi\colon \textbf{F} \otimes \textbf{F} \rightarrow S_2(\textbf{F})$ be the map of chain complexes where each $\pi_i$ is the canonical projection from $(\textbf{F} \otimes \textbf{F})_i$ to $S_2(\textbf{F})_i$. 
From the definition of $S_2(\textbf{F})$ we see that any lift $\mu\colon \textbf{F} \otimes \textbf{F} \rightarrow \textbf{F}$ of $\tilde{\mu}$ which factors through $\pi$ satisfies property \eqref{item:commut}.
\begin{proof}[Proof of Proposition \ref{prop_multiplication}]
To begin with, we set
 $\alpha_0 := \tilde{\mu} \circ (d_0 \otimes d_0)\colon F_0 \otimes F_0 \rightarrow R$.
Since $R$ is a commutative ring, 
$\alpha_0$ factors through $\pi_0$ yielding a homomorphism 
$ \gamma_0\colon S_2(\textbf{F})_0  \cong S_2(F_0) \rightarrow R$. 
Now, as $S_2(\textbf{F})_0$ is free, there is a map $\beta_0\colon S_2(\textbf{F})_0 \rightarrow F_0$ such that $ \gamma_0 = d_0 \circ \beta_0.$
From this, setting $\mu_0 = \beta_0 \circ \pi_0$, we obtain a  commutative diagram 
\begin{equation}\label{diag_comm}	
\begin{tikzpicture}[baseline=(current bounding box.east)]
  \matrix (m) [matrix of math nodes,row sep=1.9em,column sep=2.1em,minimum width=2em,text height=1.5ex, text
  depth=0.25ex]
  { & R \otimes R & & & & & F_0 \otimes F_0 \\
    &                   & & & & & \\ 
    &                   & & & & \color{red}S_2(\textbf{F})_0  &   \\
   0& R             &  & & & & F_0 
     \\};
  \path[-stealth]
    (m-1-2) edge node [left] {$\tilde{\mu}$} (m-4-2)
    (m-4-7) edge node [above] {$d_0$} (m-4-2)
    (m-1-7) edge node [above] {$d_0 \otimes d_0$} (m-1-2)
    (m-4-2) edge (m-4-1)
    (m-1-7) edge node[left=0.5em]{$\alpha_0$} (m-4-2)
    (m-1-7) edge[red] node[red,auto] {$\pi_0$} (m-3-6)
    (m-3-6) edge[red] node[red,auto] {$\gamma_0$} (m-4-2)
    (m-3-6) edge[red] node[red,auto] {$\beta_0$} (m-4-7)
    (m-1-7) edge node[left] {$\mu_0$} (m-4-7);
\end{tikzpicture}
\end{equation}
Note that for any element $g \in F_0$ we have and $\gamma_0(\pi_0(e_0 \otimes g))  = \gamma_0(\pi_0(g \otimes e_0))$ and 
\[ \gamma_0(\pi_0(e_0 \otimes g))  =\alpha_0(g \otimes e_0) = d_0(g).\]
Hence, we can choose $\beta_0$ in such a way that $ \beta_0(\pi_0(e_0 \otimes g)) = \beta_0(\pi_0(g \otimes e_0)) = g$. Consequently, 
$$ \mu_0(e_0 \otimes g) = \mu_0(g \otimes e_0) = g$$ for all $g \in F_0$.  
By the choice of $\gamma_0$, 
$$ 0 \leftarrow R \xleftarrow{\gamma_0} S_2(\textbf{F})_0 \xleftarrow{\bar{\delta}_1}  S_2(\textbf{F})_1 \xleftarrow{\bar{\delta}_2}  \ldots$$ is a complex of free $S$-modules. 

Comparing this complex $S_2(\textbf{F})$ with $\textbf{F}$ we can extend $\beta_0$ to a map of chain complexes $\beta\colon S_2(\textbf{F}) \rightarrow  \textbf{F}$:
\begin{center}
\begin{tikzpicture}
\matrix(m)[matrix of math nodes,
row sep=2.6em, column sep=2.8em,
text height=1.5ex, text depth=0.25ex]
{0 & R & S_2(\textbf{F})_0 & S_2(\textbf{F})_1 & S_2(\textbf{F})_2 & \cdots \\
0 & R &  F_0 & F_1 & F_2 & \cdots \\};
\path[->,font=\scriptsize,>=angle 90]
(m-1-2) edge (m-1-1)
        edge node[left] {$\id$} (m-2-2) 
(m-1-3) edge node[above] {$\gamma_0$} (m-1-2) 
        edge[red] node[left,red] {$\beta_0$} (m-2-3) 
(m-1-4) edge node[above] {$\bar{\delta}_1$} (m-1-3)
        edge[red] node[left,red] {$\beta_1$} (m-2-4) 
(m-1-5) edge node[above] {$\bar{\delta}_2$} (m-1-4)
        edge[red] node[left,red] {$\beta_2$} (m-2-5) 
(m-1-6) edge (m-1-5) 
(m-2-2) edge (m-2-1)
(m-2-3) edge node[above] {$d_0$} (m-2-2) 
(m-2-4) edge node[above] {$d_1$} (m-2-3)
(m-2-5) edge node[above] {$d_2$} (m-2-4)
(m-2-6) edge (m-2-5) 
;
\end{tikzpicture}
\end{center}
As above, we can choose the maps $\beta_i$ successively so that
$$ \beta_i(\pi_i(e_0 \otimes g)) = \beta_i(\pi_i(g \otimes e_0)) = g$$
for any $g \in F_i$:
Indeed,  if $g \in F_1$, then
$$  \beta_0(\bar{\delta}_1(\pi_1(e_0 \otimes g))) = \beta_0(\pi_0(\delta_1(e_0 \otimes g))) = \beta_0(\pi_0(e_0 \otimes d_1(g))) = d_1(g).$$
Hence, we can choose $\beta_1$ so that 
$ \beta_1(\pi_1(e_0 \otimes g)) = \beta_1(\pi_1(g\otimes e_0)) = g$ and we proceed similarly for $i \geq 2$. 
\par \smallskip
Setting $\mu_i = \beta_i \circ \pi_i$ for $i \geq 1$, we obtain a chain map $\mu\colon \textbf{F} \otimes \textbf{F} \rightarrow  \textbf{F}$ which factors through $\beta\colon S_2(\textbf{F}) \rightarrow \textbf{F}$ as desired. 
Finally, we summarize the constructed maps in one diagram:
\begin{center}
\begin{tikzpicture}
  \matrix (m) [matrix of math nodes, row sep=3.8em,
    column sep=1.2em,text height=1.5ex, text depth=0.25ex]{
    & & & R \otimes R& & (\textbf{F} \otimes \textbf{F})_0& & (\textbf{F} \otimes \textbf{F})_1 & & \cdots \\
    & & & & S_2(\textbf{F})_0 & & S_2(\textbf{F})_1 & & \cdots &\\
   0& & R & & F_0 & & F_1 & & \cdots &\\};
  \path[-stealth]
        (m-1-8) edge node[above] {$\delta_1$} (m-1-6) 
                edge node[left] {$\pi_1$} (m-2-7)
         (m-1-6) edge (m-3-5) 
         (m-1-6) edge (m-3-5) 
         (m-2-7) edge [-,line width=8pt,draw=white] (m-2-5) 
                edge node[above] {$\bar{\delta}_1$} (m-2-5)
                 edge node[left] {$\beta_1$} (m-3-7)
         (m-2-5) edge node[left] {$\beta_0$} (m-3-5)
         (m-3-7) edge node[above] {$d_1$} (m-3-5)
         (m-1-6) edge node[left] {$\pi_0$} (m-2-5)
         (m-1-8) edge (m-3-7) 
         (m-2-9) edge [-,line width=8pt,draw=white] (m-2-7) 
                 edge (m-2-7)
         (m-3-9) edge (m-3-7)       
         (m-1-10)edge (m-1-8)   
         (m-1-6) edge node[above] {$d_0 \otimes d_0 $} (m-1-4)
         (m-3-3) edge (m-3-1)
         (m-3-5) edge node[above] {$d_0$}(m-3-3)
         (m-2-5) edge node[above left] {$\gamma_0$} (m-3-3)
         (m-1-4) edge node[above left] {$\tilde{\mu}$} (m-3-3)
        ;
\end{tikzpicture}
\end{center} 
\end{proof}

\section{Proof of Theorem ~\texorpdfstring{\ref{thm_mainthm}}{\ref*{thm_mainthm}}}
As a next step towards proving the remaining part of Theorem \ref{thm_mainthm} we use the multiplication on $\textbf{F}$ from the last section and the Gorenstein property of $R$ to deduce an isomorphism between the given resolution of $R$ an its dual.
So let us assume that $R$ is a Gorenstein algebra of odd codimension $c = 2m+1$  
We know from the introductory part that 
 $$ 0 \leftarrow \Ext^c_S(R,S) \leftarrow F_c^\ast \xleftarrow{d_c^\ast} F_{c-1}^\ast \leftarrow \cdots \leftarrow  F_0^\ast \leftarrow 0$$
 is a minimal free resolution of $\Ext^c_S(R,S)$ 
 Furthermore, using the Gorenstein property, there exists an isomorphism $\tilde{\psi}\colon \Ext^c_S(R,S) \rightarrow R$ which lifts to an isomorphism $\psi_0\colon F_c^\ast \rightarrow F_0$.  
 Let $\sigma \in F_c^\ast$ be the preimage of $e_0 \in F_0$ under this isomorphism.
 For any $0 \leq i \leq c$ we define a map $ h_i\colon \ F_i \otimes F_{c-i} \rightarrow S$
by sending $a \otimes b$ to $\sigma(ab)$. Moreover, for each $i$, this induces a map
\begin{align*}
 t_i\colon & F_i \rightarrow  \Hom_S(F_{c-i},S) =  F_{c-i}^{\ast}\\
      & a \mapsto t_i(a) \colon  F_{c-i} \rightarrow S \\
                       & \quad  \quad \quad  \quad \quad b \mapsto h_i(a \otimes b)  = \sigma(ab).
\end{align*} 
\begin{lemma}\label{lem_mapsresdual}
In the diagram
\begin{center}
\begin{tikzpicture}
\matrix(m)[matrix of math nodes,
row sep=2.6em, column sep=2.8em,
text height=1.5ex, text depth=0.25ex]
{ F_0 & F_1 & \cdots & F_m & F_{m+1} & \cdots \\
 F_c^{\ast} & F_{c-1}^{\ast} & \cdots  & F_{m+1}^\ast & F_{m}^\ast & \cdots \\};
\path[->,font=\scriptsize,>=angle 90]
(m-1-2) edge node[above] {$d_1$}(m-1-1)
        edge node[left] {$t_1$} (m-2-2)
(m-1-3) edge (m-1-2)
(m-1-4)  edge node[above] {$d_{m}$} (m-1-3)
         edge node[left] {$t_{m}$} (m-2-4) 
(m-2-4) edge  (m-2-3)
(m-2-3) edge (m-2-2) 
(m-2-2) edge node[above] {$d_c^{\ast}$}  (m-2-1)   
(m-1-1) edge node[left] {$t_0$} (m-2-1) 
(m-1-5) edge node[above] {$d_{m+1}$} (m-1-4)
        edge node[left] {$t_{m+1}$} (m-2-5)  
(m-2-5) edge node[above] {$d_{m+1}^{\ast}$} (m-2-4)
(m-1-6) edge (m-1-5)
(m-2-6) edge node[above] {$d_{m}^\ast$}  (m-2-5)
;
\end{tikzpicture} 
\end{center}
the rectangles with homomorphism $d_i \colon F_i \rightarrow F_{i-1}$ are commutative if $i$ is an odd number and anti-commutative if $i$ is even. For the diagram in the middle we have

\[ t_m \circ d_{m+1}  = (-1)^{m} d_{m+1}^\ast \circ t_{m+1}
\]
Moreover, $t_{c-i} = t_i^\ast$.
\end{lemma}
\begin{proof}
Let $i \in \{ 1,\ldots,c\}$, and let $a \in F_i$. We have to show that $t_{i-1}(d_i(a))  = d_{c+1-i}^\ast(t_i(a_i)) \in F_{c+1-i}^\ast$. For any element $b \in F_{c+1-i}$ we have 
\begin{align*}
 t_{i-1}(d_i(a))(b) &= \sigma(d_i(a)b)    \\
 d_{c+1-i}^\ast(t_i(a))(b) & = \sigma(ad_{c+1-i}(b))
\end{align*}
On the other hand, as $ab = 0$,
$$0 = d_{i}(a)b + (-1)^{i}ad_{c+1-i}(b).$$ 
Hence, for an odd integer $i$ we have $d_i(a)b = ad_{c+1-i}(b)$, whereas for an even integer $i$ we have $d_i(a)b = -ad_{c+1-1}(b)$. This shows the first claim. Next let $i \in \{0,\ldots,m\}$, and let $a_i \in F_i$, $b_{c-i} \in F_{c-i}$. Then $t_i(a_i)(b_{c-i}) = \sigma(a_ib_{c-i})$ and $t_{c-i}(b_{c-i})(a_i) = \sigma(b_{c-i}a_i)$. 
Thus, the second claim follows from the fact that 
$a_ib_{c-i} = b_{c-i}a_i$ for every $i$.
\end{proof}
As a final step, we show that the $t_i$ are indeed isomorphisms. Note that in the setting of \cite{BuchsbaumEisenbud77}  we have $F_0 = S = F_c$ which implies directly that $t_0$ is the identity map (lifting the identity on the corresponding rings) and thus, that each $t_i$ is an isomorphism. To show this in our setting, we make now use of Proposition \ref{prop_iso}. 
\begin{proposition}\label{prop_chainmapiso}
For each $i=0,\ldots,c$, the map $t_i$ is an isomorphism.
\end{proposition}
\begin{proof}
Let $\tilde{t}\colon R \rightarrow \Ext^c_S(R,S)$ be the $S$-linear map induced by the chain maps
in Lemma \ref{lem_mapsresdual}.
Since both complexes are minimal free resolutions, it is enough to show that $\tilde{t}$ is an isomorphism. Now let $\tilde{\psi}\colon \Ext^c_S(R,S) \rightarrow R$ be the isomorphism introduced at the beginning of this section with induced isomorphism $\psi_0: F_c^\ast \rightarrow F_0$. Let us see what happens with the element $e_0 \in F_0$ in the following diagram:
\begin{center}
	\begin{tikzpicture}
	\matrix(m)[matrix of math nodes,
	row sep=2.4em, column sep=2.8em,
	text height=1.5ex, text depth=0.25ex]
	{ 0 & R & F_0 & \cdots \\
	  0 & \Ext^c_S(R,S) & F_c^\ast & \cdots \\
	  0 & R & F_0 & \cdots \\};
	\path[->,font=\scriptsize,>=angle 90]
	(m-1-2) edge (m-1-1)
	        edge node[left] {$\tilde{t}$} (m-2-2)
	(m-1-3) edge node[above] {$d_0$} (m-1-2)
	        edge node[left] {$t_0$} (m-2-3)
	(m-1-4) edge (m-1-3)
	(m-2-2) edge (m-2-1)
	        edge node[left] {$\psi$} (m-3-2)
	(m-3-2) edge (m-3-1)
	(m-3-3) edge node[above] {$d_0$} (m-3-2)
    (m-2-3) edge  (m-2-2)
            edge node[left] {$\psi_0$} (m-3-3)
    (m-2-4) edge  (m-2-3)
    (m-3-4) edge  (m-3-3)
	;
	\end{tikzpicture} 
\end{center}
 Under the homomorphism $t_0$,  the element $e_0 \in F_0$ is mapped to $\sigma$. Indeed for any $b \in F_c$ we  have 
$t_0(e_0)(b) = \sigma(e_0b) = \sigma(b)$.  Hence, by the choice of  $\sigma$ we know that $\psi_0(t_0(e_0)) = e_0$ and, consequently, $\tilde{\psi}(\tilde{t}(1_R)) = 1_R$ as $d_0(e_0) = 1_R$. Now, using Proposition \ref{prop_iso}, we know that the natural homomorphism $R \rightarrow \Hom_{\ts}(R,R)$ is an isomorphism which implies that $\tilde{\psi} \circ \tilde{t}$ is the identity on $R$. Note that both $\tilde{\psi}$ and $\tilde{t}$ are also $\ts$-linear. Thus, since $\tilde{\psi}$ isomorphism, so is $\tilde{t}$. 
\end{proof}
The proof of the remaining part of Theorem $\ref{thm_mainthm}$ is a direct consequence from the previous results:
\hspace{-1pt}
\begin{proof}[Proof of Theorem \ref{thm_mainthm}, "Gorenstein $\Rightarrow$ (skew-) symmetric resolution"]
By Proposition \ref{prop_chainmapiso}, the homomorphism  $t_{m+1}\colon F_{m+1} \rightarrow F_{m}^\ast$ is an isomorphism. Then $\widetilde{d_{m+1}} :=  d_{m+1} \circ t_{m+1}^{-1}$ is a homomorphism from $F_m^{\ast} $ to $F_m$ and, as $ t_{m+1}^{\ast} \circ d_{m+1} =  t_m\circ d_{m+1} = (-1)^m d_{m+1}^{\ast} \circ t_{m+1} $, we have:
\begin{align*}
\widetilde{d_{m+1}}^{\ast} & = (d_{m+1} \circ t_{m+1}^{-1})^{\ast} 
                    = ({t_{m+1}^{-1}})^{\ast} \circ d_{m+1}^{\ast} \\ 
                   & = (-1)^m d_{m+1}  \circ t_{m+1}^{-1} 
                    = (-1)^m \widetilde{d_{m+1}}.
\end{align*}
Consequently,
$$ 0 \leftarrow R \leftarrow F_0 \xleftarrow{d_1} F_1 \leftarrow \cdots \leftarrow F_{m} \xleftarrow{\widetilde{d_{m+1}}} F_m^\ast \leftarrow \cdots \leftarrow 
F_1^{\ast} \xleftarrow{d_1^{\ast}} F_0^{\ast} \leftarrow 0$$ 
is a minimal free resolution of $R$ with symmetric or alternating middle map as desired. 
\end{proof}
\section{The graded case}
We reformulate our structure theorem in the graded case. The proofs are entirely along the same line as in the local case and are omitted here.  
\begin{definition}
Let $S$ be a positively graded polynomial ring, and let $R$ be a finite graded $S$-algebra. By $c = \dim S - \dim_S R$ we denote the codimension of $R$. Then $R$ is called a \textit{Gorenstein algebra of codimension $c$ and twist $t \in \Z$} if 
$$ R \cong \Ext^c_S(R,S(t))$$ 
as $R$-modules.
\end{definition}
\begin{theorem}\label{thm_mainthmgraded}
Let $S$ be a positively graded polynomial ring, and let $R$ be any graded ring. Let $\varphi \colon S \rightarrow R$ be a homogeneous ring homomorphism which induces on $R$ the structure of a finitely generated graded $S$-module. Furthermore, let $\tilde{S}  = S/\ann_S R$ and $\tilde{\varphi}\colon \tilde{S} \rightarrow R$ be the induced injective homomorphism. Assume that the condition 
$$ (\diamondsuit) \text{\quad  $\tilde{\varphi}_p \colon \tilde{S}_p \rightarrow R_p$ is an isomorphism for all minimal primes of $\tilde{S}$ }$$
holds.
If the codimension 
$ c = \dim S - \dim_S R$ 
is an odd integer $2m+1$, then the following are equivalent:
\begin{itemize}
\item $R$ is a Gorenstein algebra of codimension $c$ and twist $t$.
\item There exists a minimal free resolution of $R$ as an $S$-module of type
\begin{equation}
 0 \leftarrow R \xleftarrow{d_0}
  F_0  \xleftarrow{d_1} F_1 \leftarrow \cdots \xleftarrow{d_{m}}  F_m \xleftarrow{d_{m+1}} F_m^\ast(t) \xleftarrow{d_{m}^\ast} \cdots \leftarrow F_1^\ast(t) \xleftarrow{d_1^\ast} F_0^\ast(t) \leftarrow 0,
\end{equation}
where $F_0  = S \oplus  F_0'$ for some graded free module $F_0'$, $d_0 = ( \varphi  \mid  d_0')$ and $d_{m+1}^\ast = (-1)^m d_{m+1}$.
\end{itemize}
\end{theorem}
\section{Examples}
Finally, we present some applications of our structure result for the cases $c = 1$ and $c =3$ originating from algebraic geometry.
\begin{example}
As a first example we study the canonical ring of a smooth general curve  $C$ of genus $6$. Canonically embedded, $C$ is a curve of degree 10 in $\P^5_{\langle x_0,\ldots,x_4,y \rangle}$ and a quadratic section of a del Pezzo surface $Y \subset \P^5$ of degree 5 uniquely determined by $C$. The anti-canonical ring of $Y$ is a Gorenstein ring of dimension 3 and $Y$ is defined by the submaximal Pfaffians of a skew-symmetric $5\times 5$ matrix $m$. Let $T = \Bbbk[x_0,\ldots,x_4,y]$. The canonical ring $R(C)$ is of the form $T/I$, where $I$ is a Gorenstein ideal of codimension 4, with a minimal free resolution of the form
\[0 \leftarrow R(C) \leftarrow T \leftarrow T(-2)^6 \leftarrow T(-3)^5 \oplus T(-4)^5 \leftarrow T(-5)^6 \leftarrow T(-7)^1\leftarrow 0.\]
 Let us assume that 
\[m = \begin{pmatrix}
0  & y & l_1 & l_3 & x_0 \\
-y & 0 & l_2 & l_4 & x_1 \\
-l_1  & -l_2 & 0 & y & x_2 \\
-l_3  & -l_4 & -y & 0 & x_3 \\
-x_0  & -x_1 & -x_2 & -x_3 & 0 \\ 
\end{pmatrix}\]
where $l_1,\ldots,l_4 \in \Bbbk[x_0,\ldots,x_4]$ are some non-zero linear forms, and that the quadratic section $f$ defining $C$ is of the form 
$x_4y + q$ with $q \in S:=\Bbbk[x_0,\ldots,x_4]$ being a quadratic form.   
Projecting from the point $p = (0:0:0:0:0:1)$ to $\P^4_{\langle x_0,\ldots,x_4 \rangle}$, the image of $Y$ is a surface with a unique singular point $q$, whereas $C$ is isomorphic to its image in $\P^4$. Now $R(C)$ is a finitely generated Cohen-Macaulay $S$-module and an $S$-algebra of codimension 3. Furthermore, the condition $(\diamondsuit)$ is satisfied since $\Proj(R(C))$ and $\Proj(S/\ann_S(R(C)))$ are isomorphic curves. Hence, we can apply Theorem \ref{thm_mainthmgraded} to obtain a minimal free resolution of $R(C)$ as an $S$-module of the form 
\[0 \leftarrow R(C) \leftarrow F_0 \xleftarrow{d_1} F_1 \xleftarrow{d_2} F_1^\ast(t) \xleftarrow{d_1^\ast} F_0^\ast(t) \leftarrow 0\]
with an alternating map $d_2$. Moreover, a computation shows that the minimal free resolution is of the form
\[0 \leftarrow R(C) \leftarrow 
\begin{array}{c}
S \\ \bigoplus \\ S(-1) 
\end{array} 
\xleftarrow{d_1} 
\begin{array}{c}
S(-2)^5 \\ \bigoplus \\ S(-3) 
\end{array}
\xleftarrow{d_2} 
\begin{array}{c}
S(-4)^5 \\ \bigoplus \\ S(-3) 
\end{array} 
\xleftarrow{d_1^\ast} 
\begin{array}{c}
S(-5) \\ \bigoplus \\ S(-6) 
\end{array} 
 \leftarrow 0\]
 with a possible choice of differentials
 \begin{small}
 \[ d_ 1 = \begin{pmatrix}
        {x}_{3}{l}_{1}-{x}_{2}{l}_{3}&
        {x}_{3}{l}_{2}-{x}_{2}{l}_{4} & 
        -x_1l_1+x_0l_2 &
        -x_1l_3+x_0l_4 &
         q & x_4q(l) \\
        {x}_{0}&{x}_{1}&{x}_{2}&{x}_{3}&{x}_{4}&
        q
        \end{pmatrix}\]
 \end{small}
where $q(l) = l_1l_4-l_2l_3$, and 
\[ d_2 = 
\begin{pmatrix}
 0& q&{-{x}_{4}{l}_{4}}&{x}_{4}{l}_{2}&-{x}_{3}{l}_{2}+{x}_{2}{l}_{4}&{x}_{1}\\
  & 0&{x}_{4}{l}_{3}&{-{x}_{4}{l}_{1}}&{x}_{3}{l}_{1}-{x}_{2}{l}_{3}&{-{x}_{0}}\\
   & & 0& q
   &{x}_{1}{l}_{3}-{x}_{0}{l}_{4}&{x}_{3}\\
       & & & 0&-{x}_{1}{l}_{1}+{x}_{0}{l}_{2}&{-{x}_{2}}\\
       & &-\text{skew} & & 0 & 0\\
       & & & & & 0
\end{pmatrix}.
\]

 \end{example}

In the first example, we used the (known) general description of the canonical curve $C \subseteq \P^5$
to deduce a minimal free resolution of $R(C)$ as an $S$-module. However, the strength of our structure result lies in a reversed approach, that means constructing an (unknown) variety via some projection to a lower dimensional projective space. We will sketch this approach in the next examples.

\begin{example}
	As an example for codimension $c = 1$ we consider a minimal surface $X$ of general type with $K^2 = 6$, $p_g(X) = h^0(X,K_X) = 4$ and $q(X) = h^1(X,\Osh_X) = 0$ whose canonical system $|K_X|$ has no base points and defines a birational morphism. These surfaces were completely described in \cite{CataneseRegular} by the following method. First, choose a minimal set of algebra generators of the canonical ring $R(X)$. 
	Let $u_0,\ldots,u_3$ be a basis of $H^0(X,K_X)$. Now $h^0(X,2K_X) = K_X^2+ \chi = 11$ implies that we need one further generator in degree 2, denoted by $x_0$. Now let us consider $R(X)$ as an $S := \Bbbk[u_0,\ldots,u_3]$- module, and let $\varphi\colon S \rightarrow R(X)$ be the natural ring homomorphism.  Since the morphism of projective schemes induced by $\varphi$ is birational onto its image, the morphism $\varphi$ satisfies condition $(\diamondsuit)$. Hence, there exists a minimal free resolution of the type 
	$$ 0 \leftarrow R(X) \leftarrow F_0 \xleftarrow{d_1} F_0^\ast(t) \leftarrow 0,$$
	where $t = -5$ and $F_0 = S \oplus S(-2)$. Thus, as a starting point for constructing a surface $X$ with the given invariants we choose a symmetric matrix 
	$$ d_1 = \begin{pmatrix}
	a_{11} & a_{12} \\
	a_{12} & a_{22}
	\end{pmatrix}$$
	whose entries are homogeneous polynomials in $S$ 
    with $\deg(a_{11}) = 5, \deg(a_{12}) = 3$ and $\deg(a_{22}) = 1$. Furthermore, we choose the entries so that $f = \det(d_1)$ defines an irreducible surface in $\P^3$. In \cite{CataneseRegular}, Catanese gives a necessary and sufficient condition, denoted as the rank condition, for $R := \coker d_1$ carrying the structure of a commutative ring:
    \par
    \begingroup
    \leftskip=6cm
   \noindent
    \emph{
	(R.C.) Let $h$ be the number of rows of $d_1$, and let $d_1'$ be the matrix obtained by deleting the first row of $d_1$. Then the ideal generated by the entries of $\bigwedge^{h-1}d_1$ coincides with the ideal generated by the entries of $\bigwedge^{h-1}d_1'$}.
    \endgroup
    \par 
 In this example, the rank condition is satisfied if $a_{11} = c_2a_{12} + c_4{a_{22}}$, where $c_2$ is a quadratic form in $S$ and $c_4$ a quartic form. Furthermore, if (R.C.) is satisfied, we can compute the remaining defining relations of $R(X)$ as an $\Bbbk$-algebra and see that canonical model $\Proj(R(X))$ can be embedded $\P(1^4,2)$ as an complete intersection defined by the relations
\begin{align*}
	a_{1,2}+a_{2,2}x_0 &= 0, \\
	x_0^2 - c_2x_0+c_4 & = 0.
\end{align*}
\end{example}

\begin{example}
 Describing the canonical ring of a numerical Godeaux surface has served as a main motivation for the presented structure result.
 This description is a crucial tool in the construction of numerical Godeaux surfaces in \cite{Schreyer05} and \cite{Stenger18}. Recall that a numerical Godeaux surface $X$ is a minimal surface of general type with $K_X^2 =1$ and $p_g(X) = q(X) = 0$.  The canonical ring $R(X)$ is a finitely generated positively graded $\Bbbk=\C$-algebra. Studying its generators and relations, one can deduce that $R(X)$ is of the form \begin{equation}\label{eq_canringgod}
\Bbbk[x_0,x_1,y_0,\ldots,y_3,z_0,\ldots,z_3,w_0,w_1,w_2]/I,
 \end{equation}
  where $\deg(x_i) =2$, $\deg(y_j) = 3$ and $\deg(z_j) = 4$ and $\deg(w_k) = 5$ and $I$ is a Gorenstein ideal of codimension 10 (for more details and proofs we refer to \cite{Stenger18}, Chapter 3). Thus, the minimal free resolution of $R(X)$ has length 10 and finding a general description seems hopeless. The idea is to study to $R(X)$ as an graded $S := \Bbbk[x_0,x_1,y_0,\ldots,y_3]$-module, or geometrically, via the projection from the canonical model $\Proj(R(X))$ to $\P(2^2,3^4)$. 
\begin{proposition}
Let $X$ be a numerical Godeaux surface, and let $S$, $R(X)$ and $\varphi$ be as above. Then $R(X)$ is a Gorenstein $S$-algebra of codimension 3 and twist -17 satisfying $(\diamondsuit)$.  
\end{proposition}
\begin{proof} See \cite{Stenger18}, Chapter 4. 
\end{proof}
\begin{remark} The canonical ring $R(X)$ is a finitely generated $S$-module since the bi- and tricanonical system of $X$ do not have common base point by some classical results of Miyaoka (\cite{Miyaoka}). Moreover, $(\diamondsuit)$ is satisfied because the tricanonical system $|3K_X|$ induces a birational map onto its image in $\P^3$, and so does the morphism $\tilde{\varphi}\colon \Proj(R(X)) \rightarrow \Proj(S)$ with image $\Proj(S/\ann_S(R(X)))$.
\end{remark}
\begin{corollary} Let $X$ be a numerical Godeaux surface. 
Then there exists a minimal free resolution of $R(X)$ as an $S$-module of type
\[ 0 \leftarrow R(X) \leftarrow F_0 \xleftarrow{d_1} F_1 \xleftarrow{d_2} F_1^\ast(-17) \xleftarrow{d_1^\ast} F_0^\ast(-17) \leftarrow 0,\]
where $d_2$ is alternating.
\end{corollary}

\begin{remark} The construction method for numerical Godeaux surfaces from \cite{Schreyer05} and \cite{Stenger18}  originally focuses on constructing the projected surface of codimension 3, or equivalently,  an $S$-module $R$ by constructing maps $d_1$ and $d_2$.
	In \cite{Stenger18}, we introduce a sufficient condition on the entries of the matrix $d_1$ for $R$ carrying the structure of a commutative algebra.  If this condition is satisfied, then we can recover the ring structure of $R(X)$, that means all its defining relations in \eqref{eq_canringgod}, from its structure as an $S$-module using Diagram \eqref{diag_comm}. Thus, our structure result gives a powerful tool for constructing (the canonical model) of a numerical Godeaux surface from a suitable model in codimension 3.  For more details, we refer to \cite{Stenger18}, Chapter 5. 
\end{remark}

\end{example}

\noindent \textit{Acknowledgments.} The author would like to thank Frank-Olaf Schreyer for the helpful discussions and comments. This work has been supported by the  German Research Foundation (DFG), TRR  195  "Symbolic  tools  in
mathematics and their applications".

\end{document}